\documentclass[12pt]{amsart}
\usepackage{amscd,amsmath,amsthm,amssymb}
\usepackage{amsfonts,amssymb,amscd,amsmath,enumerate,verbatim}
\usepackage[left]{lineno}
\usepackage{pstcol,pst-plot,pst-3d}
\usepackage[bookmarksnumbered, colorlinks, plainpages]{hyperref}
\usepackage{tikz}
\usepackage{doi}

\newpsstyle{fatline}{linewidth=1.5pt}
\newpsstyle{fyp}{fillstyle=solid,fillcolor=verylight}
\definecolor{verylight}{gray}{0.97}
\definecolor{light}{gray}{0.9}
\definecolor{medium}{gray}{0.85}
\definecolor{dark}{gray}{0.6}

\usepackage{scalerel}

\def\NZQ{\mathbb}               
\def\NN{{\NZQ N}}

\def\ZZ{{\NZQ Z}}

\def\KK{{\NZQ K}}

\def\frk{\mathfrak}               

\def\mm{{\frk m}}

\def\G{{\mathcal G}}
\def\F{{\mathcal F}}

\def\H{{\mathcal H}}

\def\tb{{\mathbf t}}

\def\opn#1#2{\def#1{\operatorname{#2}}} 
\opn\chara{char} \opn\length{\ell} \opn\pd{pd} \opn\rk{rk}
\opn\projdim{proj\,dim} \opn\injdim{inj\,dim} \opn\rank{rank}
\opn\depth{depth} \opn\grade{grade} \opn\height{height}
\opn\embdim{emb\,dim} \opn\codim{codim}
\opn\Cl{Cl}

\opn\Tr{Tr} \opn\bigrank{big\,rank}
\opn\superheight{superheight}\opn\lcm{lcm}
\opn\trdeg{tr\,deg}
	\opn\reg{reg} \opn\lreg{lreg} \opn\ini{in} \opn\lpd{lpd}
	\opn\size{size} \opn\sdepth{sdepth}
	\opn\link{link}\opn\fdepth{fdepth}\opn\lex{lex}
	\opn\tr{tr}
	\opn\type{type}
	\opn\gap{gap}
	\opn\arithdeg{arith-deg}
	\opn\revlex{revlex}
	\opn\div{div} \opn\Div{Div} \opn\cl{cl} \opn\Cl{Cl}
	\opn\Spec{Spec} \opn\Supp{Supp} \opn\supp{supp} \opn\Sing{Sing}
	\opn\Ass{Ass} \opn\Min{Min}\opn\Mon{Mon} \opn\V{V}   \opn\k{K} 
	\opn\Ann{Ann} \opn\Rad{Rad} \opn\Soc{Soc}
	\opn\Im{Im} \opn\Ker{Ker} \opn\Coker{Coker} \opn\Am{Am}
	\opn\Hom{Hom} \opn\Tor{Tor} \opn\Ext{Ext} \opn\End{End}
	\opn\Aut{Aut} \opn\id{id}
	\def\F{{\mathcal F}}
	\opn\nat{nat}
	\opn\pff{pf}
	\opn\Pf{Pf} \opn\GL{GL} \opn\SL{SL} \opn\mod{mod} \opn\ord{ord}
	\opn\Gin{Gin} \opn\Hilb{Hilb}\opn\sort{sort}
	\opn\PF{PF}\opn\Ap{Ap}
	\opn\mult{mult}
	\opn\bight{bight}
	\opn\div{div}
	\opn\Div{Div}
	\opn\aff{aff}
	\opn\relint{relint} \opn\st{st}
	\opn\lk{lk} \opn\cn{cn} \opn\core{core} \opn\vol{vol}  \opn\inp{inp} \opn\nilpot{nilpot}
	\opn\link{link} \opn\star{star}\opn\lex{lex}\opn\set{set}
	\opn\width{wd}
	\opn\Fr{F}
	\opn\QF{QF}
	\opn\G{G}
	\opn\type{type}\opn\res{res}
	\opn\conv{conv}
	\opn\Deg{Deg}
	\opn\Sym{Sym}
	\opn\Con{Con}
	\opn\gr{gr}
	
	\def\pot#1#2{#1[\kern-0.28ex[#2]\kern-0.28ex]}

	\opn\dirlim{\underrightarrow{\lim}}
	\opn\inivlim{\underleftarrow{\lim}}
	\let\to=\rightarrow
	
	\def\Implies{\ifmmode\Longrightarrow \else
		\unskip${}\Longrightarrow{}$\ignorespaces\fi}
	\def\implies{\ifmmode\Rightarrow \else
		\unskip${}\Rightarrow{}$\ignorespaces\fi}
	\def\iff{\ifmmode\Longleftrightarrow \else
		\unskip${}\Longleftrightarrow{}$\ignorespaces\fi}

	\let\:=\colon
	\newtheorem{Theorem}{Theorem}[section]
	\newtheorem{Lemma}[Theorem]{Lemma}
	\newtheorem{Corollary}[Theorem]{Corollary}
	\newtheorem{Proposition}[Theorem]{Proposition}
	\newtheorem{Remark}[Theorem]{Remark}
	
	\newtheorem{Example}[Theorem]{Example}
	
	\newtheorem{Definition}[Theorem]{Definition}

	\let\epsilon\varepsilon
	\let\kappa=\varkappa
	\def\qed{\ifhmode\textqed\fi
		\ifmmode\ifinner\quad\qedsymbol\else\dispqed\fi\fi}
	\def\textqed{\unskip\nobreak\penalty50
		\hskip2em\hbox{}\nobreak\hfil\qedsymbol
		\parfillskip=0pt \finalhyphendemerits=0}
	\def\dispqed{\rlap{\qquad\qedsymbol}}
	
	\opn\dis{dis}
	\def\pnt{{\raise0.5mm\hbox{\large\bf.}}}
	
	\opn\Lex{Lex}
	
		\opn\KT{\k^{{\scaleto{\tb \mathstrut}{8pt}}}_{{\scaleto{\V\mathstrut}{6pt}}}}

	\opn\KTT{\k^{{\scaleto{\tb \mathstrut}{8pt}}}_{{\scaleto{[\V]\mathstrut}{6pt}}}}	

\begin{document}
		\author[A. Musapa\c{s}ao\u{g}lu, M. Nasernejad and  A. A. Qureshi]{Asl\i\ Musapa\c{s}ao\u{g}lu$^{1}$ , Mehrdad ~ Nasernejad$^{2}$ and   Ayesha Asloob Qureshi$^{1}$}
		\title[-] {The edge ideals of $\tb$-spread $d$-partite hypergraphs}

		\thanks{Ayesha Asloob Qureshi and Asl\i\ Musapa\c{s}ao\u{g}lu is supported by The Scientific and Technological Research Council of Turkey - T\"UB\.{I}TAK (Grant No: 122F128).}

		\thanks{E-mail addresses:  m\_{n}asernejad@yahoo.com , aqureshi@sabanciuniv.edu  and atmusapasaoglu@sabanciuniv.edu}  
		
		\maketitle
		
		\begin{center}
			{\it
			$^1$Sabanci University, Faculty of Engineering and Natural Sciences, \\
				Orta Mahalle, Tuzla 34956, Istanbul, Turkey\\
					$^2$Univ. Artois, UR 2462, Laboratoire de Math\'{e}matique de  Lens (LML),\\
				F-62300 Lens, France  \\
			}
		\end{center}
		
		\begin{abstract}
			Inspired by the definition of $\tb$-spread monomial ideals, in this paper, we introduce $\tb$-spread $d$-partite hypergraph $\KT$ and study its edge ideal $I(\KT)$. We prove that $I(\KT)$ has linear quotients, all powers of $I(\KT)$ have linear resolution and the Rees algebra of $I(\KT)$ is a normal Cohen-Macaulay domain. It is also shown that $I(\KT)$ is normally torsion-free and a complete characterization of  Cohen-Macaulay $S/I(\KT)$ is given. 
		\end{abstract}
		\vspace{0.4cm}
		

\section{Introduction}

In \cite{EHQ}, the third author together with Ene and Herzog introduced the notion of $t$-spread monomials in a polynomial ring $S=\KK[x_1, \ldots, x_n]$ over a field $\KK$ and studied some classes of ideals and $\KK$- algebras generated by $t$-spread monomials. Let $u=x_{i_1}\cdots x_{i_d}$ be a monomial in $S$ and $t \geq 0$. The monomial $u$ is called $t$-spread if $i_j-i_{j-1} \geq t$ for all $j=2, \ldots, d$. A monomial ideal $I \subset S$ is called $t$-spread if it is generated by $t$-spread monomials.  Any monomial ideal  in $S$ can be viewed as $0$-spread and any square-free monomial ideal as $1$-spread. After their first appearance in 2019, different classes of $t$-spread monomial ideals have been studied by many authors and recently in 2023, Ficarra gave a more generalized notion of $t$-spread monomials by replacing the integer $t$ with $\tb=(t_1, \ldots, t_{d-1}) \in \NN^{d-1}$, (see \cite{AF} and the reference therein). 

 In this paper, we study $\tb$-spread monomial ideals which appear as the edge ideals of certain $d$-partite hypergraphs. Let $V=\{V_1, \ldots, V_d\}$ be a partitioning of a finite set $U \subset \NN$ such that $p<q$ if $p \in V_i, q\in V_j$ with $i<j$. We call $\{i_1, \ldots, i_d\} \subset U$ a $\tb$-spread set if $i_j \in V_j$ for all $j=1, \ldots, d$ and $i_{j}-i_{j-1} \geq t_{j-1}$ for all $j=2, \ldots, d$. We call the hypergraph $\KT$ on vertex set $V(\KT)=U$, a complete $\tb$-spread $d$-partite hypergraph if all $\tb$-spread sets of $U$ are the edges of $\KT$. For $\tb=(1, \ldots, 1)$, the hypergraph $\KT$ is a complete $d$-partite hypergraph, see \cite[Example 3]{B}. The edge ideal of $\KT$, denoted by $I(\KT)$, is a $t$-spread monomial ideal generated by those monomials whose indices correspond to the edges of $\KT$. It turns out that $I(\KT)$ admits many nice algebraic and homological properties. It is shown in Theorem~\ref{thm:linearquotient} that $I(\KT)$ has linear quotients. The ideals with linear quotients were first defined by Herzog and Takayama in \cite{HT} and their free resolutions were computed as iterated mapping cones. Using the description of Betti numbers of ideals with linear quotients given in \cite{HT}, in Proposition~\ref{prop:set}, we provide an intrinsic way to compute Betti numbers of $I(\KT)$.

 In Section~\ref{powers}, we study the powers and fiber cone of $I(\KT)$. One of the main results of Section~\ref{powers} is given in \\
 
{\bf Corollary 3.7}{ \it The ideal $I(\KT)$ satisfies the strong persistence property and all powers of $I(\KT)$ have linear resolution.}\\

To prove Corollary~\ref{strong}, we first show that minimal generating set of $I(\KT)$ is sortable and $I(\KT)$ satisfies the $\ell$-exchange property with respect to sorting order, see Proposition~\ref{prop:Sortableset} and Theorem~\ref{lexchangepropertyissatisfied}. Then it follows from classical results of Fr\"oberg \cite{RF}, Sturmfels \cite{S} and Hochster \cite{MH} that the  Rees algebra $\mathcal{R}(I(\KT))$ is a normal Cohen-Macaulay domain, see Corollary~\ref{Cor2:CMdomain}. Then Corollary~\ref{strong} is obtained as an application of \cite[Corollary 1.6]{HQ} and  \cite[Corollary 10.1.8]{HH1}. We also compute the Krull dimension of fibercone $\mathcal{R}(I(\KT))/ \mm \mathcal{R}(I(\KT))$ which provides the limit depth of $S/ I(\KT)$ in Theorem~\ref{thm:limitdepth}.
 
Let $\H$ be a hypergraph with vertex set $V(\H)$. A set $T \subset V(\H)$ is called a {\em transversal} of $\H$ if it meets all the edges of $\H$ and the family of all minimal transversals of $\H$ is called the {\em transversal hypergraph} of $\H$, see \cite[Chapter 2]{B}. The minimal transversals of a hypergraph  $\H$ correspond to the minimal prime ideals of the edge ideal of $\H$. In Section~\ref{sec:CM}, we consider $\KT$ with $V=\{V_1, \ldots, V_d\}$ such that each $V_i$ is an interval of integers. The description of the minimal primes of $I(\KT)$ is obtained by computing the minimal generating set of Alexander dual of $I(\KT)$ in Theorem~\ref{TH.ASSPRIMES}. In Theorem~\ref{Normally torsion-freeness}, we prove that $I(\KT)$ is normally torsion-free which is equivalent to say that $\KT$ is a Mengerian hypergraph. A complete characterization of unmixed $I(\KT)$ is given in Theorem~\ref{thm:unmixed}. With the help of Theorem~\ref{thm:unmixed}, a complete characterization of Cohen-Macaulay $S/I(\KT)$ is obtained in Theorem~\ref{CM}.


		\section{$\tb$-spread $d$-partite hypergraphs and their edge ideals}\label{prem}
		
A finite {\it hypergraph} $\mathcal{H}$ on the vertex set $V({\mathcal{H}})=[n]$ is a collection of edges  $E({\mathcal{H}})=\{ E_1, \ldots, E_m\}$ with $E_i \subseteq V({\mathcal{H}})$  for all $i=1, \ldots,m$. A hypergraph $\mathcal{H}$ is called {\it simple}, if $E_i \subseteq  E_j$ implies  $i = j$. Simple hypergraphs are also known as {\em clutters}. Moreover, if $|E_i|=d$, for all $i=1, \ldots, m$, then $\mathcal{H}$ is called a {\em $d$-uniform} hypergraph. A $2$-uniform hypergraph $\mathcal{H}$ is just a finite simple graph. A vertex of a  hypergraph $\mathcal{H}$ is  said  to be an {\em isolated vertex}  if it is not contained in any edge of $\mathcal{H}$.

A hypergraph $\mathcal{H}$ is a $d$-partite hypergraph if its vertex set $V(\mathcal{H})$ is a disjoint union of sets $V_1, \ldots, V_d$ such that if $E$ is an edge of $\mathcal{H}$, then $|E \cap V_i | \leq 1$. In particular, if $\mathcal{H}$ is a $d$-uniform $d$-partite hypergraph with a vertex partition $V_1, \ldots, V_d$, then $|E|=d$ and $|E \cap V_i | =1$ for each $E \in E(\mathcal{H})$. In this paper, all hypergraphs are simple, uniform, and  without isolated vertices. 
		
Next, we introduce the definition of $\tb$-spread $d$-partite hypergraphs. To do this, we give the following notation. For any integers $i\leq j$, let $[i, j]:=\{k: i \leq k \leq j\}$ and for any integer $n$, we set $[n]:=\{1,\ldots, n\}$. 


		\begin{Definition}\label{def:tspread dpartite}
			Let $\H$ be a $d$-partite hypergraph with $V(\H)\subseteq  [n]$, and $V=\{V_1, \ldots, V_d\}$ be a family defining partitioning of $V(\H)$ such that if $p\in V_i$ and $q \in V_j$ with $i <j$,  then $p<q$.
			Let ${\tb}=(t_1, \ldots, t_{d-1}) \in \NN^{d-1}$. An edge $E$ of $\H$ is called a ${\tb}$-spread edge if
			\begin{enumerate}
				\item[($*$)]$E=\{i_1, i_2, \ldots, i_d\}$  with $i_j \in V_j$ for all $j=1, \ldots, d$, and $i_j -i_{j-1} \geq t_{j-1}$ for all $j=2, \ldots, d$.
			\end{enumerate}
			A $d$-partite hypergraph $\H$ is called ${\tb}$-spread  if each edge of $\H$ is $\tb$-spread. Moreover, $\H$ is called a complete $\tb$-spread $d$-partite hypergraph and denoted by $\KT$ if all $E\subseteq  V(\H)$ satisfying $(*)$ belong to $E(\H)$.
		\end{Definition}


		Let ${\bf 1}= (1, \ldots, 1)$. A complete ${\bf 1}$-spread $d$-partite hypergraph is just a complete $d$-partite hypergraph as studied in \cite{B}. The class of complete $d$-partite hypergraphs have many nice combinatorial properties. We refer reader to \cite{B} for more information.  

Let $S=\KK[x_1, \ldots, x_n]$ be a polynomial ring over a field $\KK$ and $I$ be a monomial ideal in $S$. 
		Throughout the following text, the unique minimal generating set of a monomial ideal $I$ will be denoted by $\G(I)$. The {\em support} of a monomial $u$, denoted by $\mathrm{supp}(u)$, is the set of variables that divide $u$. Moreover, we set $\mathrm{supp}(I)=\bigcup_{u \in \G(I)}\mathrm{supp}(u)$.  Let $\H$ be a hypergraph on $V(\H)=[n]$. The {\it edge ideal} of $\mathcal{H}$ is given by
		$$I(\mathcal{H}) = (\prod_{j\in E_i} x_j : E_i\in  E({\mathcal{H}})).$$
		

		\begin{Definition}\cite{AF} Let $\tb=(t_1,t_2,\ldots,t_{d-1}) \in \NN^{d-1}$. A monomial $x_{i_1} x_{i_2} \cdots x_{i_d} \in S= \KK[x_1, \ldots, x_n]$ with $i_1 \leq i_2 \leq \cdots  \leq i_d$ is called {\it $\tb$-spread} if $i_j -i_{j-1} \geq t_{j-1}$ for all $j=2, \ldots, d$.  A monomial ideal in $S$ is called a {\it $\tb$-spread monomial ideal} if it is generated by $\tb$-spread monomials. 
		\end{Definition}
	

		Note that a ${\bf 0}$-spread monomial ideal is just an ordinary monomial ideal, while a ${\bf 1}$-spread monomial ideal is just a square-free monomial ideal. When $\tb=(t, \ldots, t)$ for some fixed integer $t \geq 0$, then $\tb$-spread monomial ideal is $t$-spread introduced in \cite{EHQ}. In the following text, we will assume that $t_i \geq 1$ for all $1 \leq i \leq d-1$. It follows from the above definitions that the edge ideal of a  $\tb$-spread $d$-partite hypergraph is a $\tb$-spread monomial ideal. To illuminate these definitions, we provide the following example. 
	

		\begin{Example}
	Let $\tb=(3,2,4)$ and $\mathrm{V}=\{V_1, V_2, V_3, V_4\}$ with $V_1=\{1,2,3\}, V_2=\{5,7\}, V_3=\{8,9,11\}$ and $V_4=\{12,13\}$. Then the minimal generators of the edge ideal of $\KT$  are  as follows:
	
	\begin{center}
		\begin{tabular}{ c c c }
			$x_1x_5x_8x_{12}$\quad&\quad$x_2x_5x_8x_{12}$\\
			$x_1x_5x_8x_{13}$\quad&\quad$x_2x_5x_8x_{13}$\\
			$x_1x_5x_9x_{13}$\quad&\quad$x_2x_5x_9x_{13}$\\
			$x_1x_7x_9x_{13}$\quad&\quad$x_2x_7x_9x_{13}$\quad&\quad$x_3x_7x_9x_{13}$
			
		\end{tabular}
	\end{center}
 The ambient ring of $I(\KT)$ in this case is $S=\KK[x_1,x_2,x_3, x_5, x_7, x_8,x_9, x_{12},x_{13}]$. Indeed, we can remove $11$ from $V_3$ to exclude the isolated vertices.
		\end{Example}
	

		The edge ideals of $\KT$ have many nice algebraic and combinatorial properties. Let $I$ be a homogenous ideal in $S=\KK[x_1, \ldots, x_n]$ with graded minimal free resolution
		\begin{equation}\label{resolution}
			0 \rightarrow \mathbb{F}_p \xrightarrow[]{\phi_p } \mathbb{F}_{p-1} \rightarrow \cdots \rightarrow  \mathbb{F}_1 \xrightarrow[]{\phi_1} \mathbb{F}_0 \xrightarrow[]{\phi_0}  I \rightarrow 0,
		\end{equation}
		where for all $i=0, \ldots, p$, the free $S$-module $\mathbb{F}_i$ is equal to $\bigoplus_j S(-j)^{\beta_{i,j}(I)}$. Recall that
		$\beta_{i,j}(I)$ is the {\it $(i,j)$-th graded Betti number of $I$} and the rank of $ \mathbb{F}_i$ is called the {\it $i$-th Betti number of $I$} and denoted by $\beta_i(I)$. Then the  ideal $I$ is said to have {\em $d$-linear} resolution if  $\beta_{i,j}(I)=0$ for all $i$ and all $j \neq d$.

		We first prove that $I(\KT)$ has linear resolution. To do this, we show that $I(\KT)$ has linear quotients. Recall that an ideal $I\subset S=\KK[x_1, \ldots, x_n]$ is said to have {\em  linear quotients}   if $\G(I)$ admits an ordering $u_1, \ldots, u_r$ such that the colon ideal $(u_1, \ldots, u_{i-1}):(u_i)$ is generated by variables for all $i=2, \ldots, r$.  It is known from \cite[Theorem 1.12]{HT} or \cite[Propositon 8.2.1]{HH1} that an ideal generated in a single degree has linear resolution if it admits linear quotients. 
	

		\begin{Theorem}\label{thm:linearquotient}
			The ideal $I(\KT)$ has linear quotients. 
		\end{Theorem}
		\begin{proof}
			Let $>_{\lex}$ denote the lexicographical order induced by the total order $x_1 > x_2 > \cdots > x_n$. Furthermore, let ${\tb}=(t_1, \ldots, t_{d-1})\in \NN^{d-1}$ and set $I=I(\KT)$ and let $\G(I)=\{u_1, \ldots u_r\}$ ordered such that $u_1 >_{\lex} u_2 >_{\lex} \cdots >_{\lex} u_r$. We need to show that $(u_1, \ldots, u_{i-1}):(u_i)$ is generated by variables for all $i=2, \ldots, r$. To do this, it is enough to show that for all $1 \leq j \leq i-1$, there exists $x_p \in (u_1, \ldots, u_{i-1}):(u_i)$ such that $x_p$ divides $u_j/ \gcd(u_j,u_i)$.

			Let $j < i$ and $u_i = x_{i_1}x_{i_2}\cdots x_{i_d}$ and $u_j =x_{j_1}x_{j_2}\cdots x_{j_d} $ with $i_1< i_2 <\cdots < i_d$ and
	 $j_1 < j_2 < \cdots < j_d$.  On account of  $u_j >_{\lex} u_i$, there exists some $\ell$ such that $j_1 = i_1,j_2 = i_2, \ldots,j_{\ell-1} = i_{\ell-1}$ and $j_{\ell} < i_{\ell}$. Note that $j_{\ell} , i_{\ell} \in V_{\ell}$. Let $v=x_{j_\ell}(u_i/x_{i_\ell})=x_{i_1}x_{i_2}\cdots x_{i_{\ell-1}} x_{j_{\ell}}x_{i_{\ell+1}} \cdots x_{i_d}$. We have $j_{\ell} - i_{\ell-1} = j_{\ell} - j_{\ell-1} \geq t_{\ell -1}$ and $i_{\ell+1} - j_{\ell} \geq i_{\ell+1} - i_{\ell} \geq  t_{\ell}$. This shows that $v$ corresponds to a ${\tb}$-spread edge of $\KT$. Hence,  $v \in \G(I)$ and $v=u_k$ for some $k<i$. This completes the proof because $x_{j_{\ell}}\in (u_1, \ldots, u_{i-1}):(u_i)$ and $x_{j_{\ell}}$ divides $u_j/ \gcd(u_j,u_i)$.
		\end{proof}
			

		Let $I$ be a monomial ideal with linear quotients with respect to the ordering $u_1, \ldots, u_r$ of $\G(I)$. If $I$ is generated in a single degree $d$, 
	 then $I$ has linear resolution as shown in \cite{HT}. Following \cite{HT}, we define 
		\[
		\set(u_k)=\{i: x_i \in (u_1, \ldots, u_{k-1}):(u_k)\} \text{ for } k=2, \ldots, r.
		\]
	 Using \cite[Lemma 1.5]{HT}, we  can conclude that  
		\[
		\beta_{i,i+d}(I)=|\{\alpha \subseteq \mathrm{set}(u): u \in \G(I) \text{ and } |\alpha|=i\}|.
		\]
	

		In the following proposition,  we give a description of $\set(u)$ when $u \in \G(I(\KT))$.  For any $S \subseteq  [n]$, we set $\min S$ to be the smallest integer in $S$,  and $\max S$ to be the largest integer in $S$. 
	

		\begin{Proposition}\label{prop:set}
			Let $u=x_{k_1}x_{k_2}\cdots x_{k_d} \in \G(I(\KT))$ with $\tb =(t_1,t_2,\ldots,t_{d-1})$ and $i_1= \min V_1$. With the notations introduced above, $\set(u)$
			is the union of $[i_1, k_1-1]\cap V_1$ and
			$
			[k_{j-1}+t_{j-1}, k_{j}-1]\cap V_{j} \text{ for } j=2, \ldots, d.
			$
		\end{Proposition}
		
		\begin{proof} 
	Let $\ell \in \set(u)$. Following Theorem~\ref{thm:linearquotient}, there exists $v \in \G(I(\KT))$ such that $v>_{\lex} u$ and $(v):(u)= (x_{\ell})$. This gives $v=(u / x_{k_j}) x_{\ell}$ for some $1 \leq j \leq d$ and $x_{k_j}, x_{\ell} \in V_j$. Since $v>_{\lex} u$, we must have $\ell \leq  k_j-1$. If $j =1$,  then $\ell \in [i_1,k_1-1]$. Moreover, if $2 \leq j \leq d$, then $k_{j-1}+t_{j-1} \leq \ell$ because $v$ is a $\tb$-spread monomial,  and hence $\ell \in 	[k_{j-1}+t_{j-1}, k_{j}-1]\cap V_{j} $. \par 
 On the other hand, if $\ell \in [i_1, k_1-1]\cap V_1$ or  $\ell \in [k_{j-1}+t_{j-1}, k_j-1]\cap V_j$ for any $j=2, \ldots, d$, then set $v=(u/x_{k_j})x_\ell$ for all $j=1, \ldots, d$. In both cases, $v \in \G(I(\KT))$ and  $v >_{\lex} u$. Therefore, $x_{\ell} \in  (v):(u)$, and hence $\ell \in \set(u)$, as required.
		\end{proof}


\section{The powers and the fiber cone of $I(\KT)$} \label{powers}

Let $\KK$ be a field and $S_d$ be the $\KK$-vector space generated by all monomials of degree $d$ in the polynomial ring $S=\KK[x_1, \ldots, x_n]$. Let $u,v \in S_d$ and $uv = x_{i_1} x_{i_2} \cdots x_{i_{2d}}$ with $i_1 \leq i_2 \leq \cdots \leq i_{2d-1} \leq i_{2d}$. Set $u'=x_{i_1} x_{i_3}  \cdots x_{i_{2d-1}}$ and $v'=x_{i_2} x_{i_4}  \cdots x_{i_{2d}}$. The map
		\begin{equation*}
			\mathrm{sort}: S_d \times S_d \rightarrow S_d \times S_d \text{ which maps } (u,v) \mapsto (u',v'), 
		\end{equation*}
		is called the {\em sorting operator}. A pair $(u,v) \in S_d \times S_d$ is called {\em sorted} if $\mathrm{sort}(u,v)=(u',v')$. A subset $A \subset S_d$ is called {\em sortable} if $\mathrm{sort}(A \times A) \subseteq  A \times A$. Furthermore, an $r$-tuple of monomials $(u_1, \ldots, u_r) \in S_d^r$ is called sorted if for any $1 \leq i < j \leq n$, the pair $(u_i,u_j)$ is sorted. In other words, if we write the monomials $(u_1, \ldots,u_r)$ as $u_1=x_{i_1}\cdots x_{i_d}$, $u_2= x_{j_1}\cdots x_{j_d} ,\; \ldots \;, u_r=x_{l_1}\cdots x_{l_d}$, then $(u_1, \ldots, u_r)$ is sorted if and only if
		\begin{equation}\label{Eq:sortingcondition}
			i_1 \leq j_1 \leq \cdots \leq l_1 \leq i_2 \leq j_2 \leq \cdots \leq l_2 \leq \cdots \leq i_d \leq j_d \leq \cdots \leq l_d.
		\end{equation}
	

		\begin{Proposition}\label{prop:Sortableset}
			The set $\G(I(\KT))$ is sortable. 
		\end{Proposition}
		
		\begin{proof}
Assume that  $u,v \in \G(I(\KT))$ and $uv=x_{i_1} x_{i_2}x_{i_3}x_{i_4}\cdots x_{i_{2d-1}} x_{i_{2d}}$ with $i_1\leq i_2 \leq  \cdots \leq i_{2d}$. Since $\supp(u)$ and $\supp(v)$ correspond to the edges of $\KT$, it follows that $i_1,i_2 \in V_1, i_3,i_4 \in V_2, \ldots , i_{2d-1},i_{2d} \in V_d$. Consequently, $u'=x_{i_1}x_{i_3} \cdots x_{i_{2d-1}} $ and $v'=x_{i_2}x_{i_4} \cdots x_{i_{2d}} $ are monomials associated to the edges of a complete $d$-partite hypergraph. It only remains to show that $u'$ and $v'$ are $\tb$-spread. We show that $u'$ is a $\tb$-spread monomial and the argument for $v'$ follows in a similar fashion. For any $1\leq l \leq d-1$, we have $i_{2l-1} \leq i_{2l} \leq i_{2l+1}$ and at least two of the variables among $x_{i_{2l-1}}, x_{i_{2l}}, x_{i_{2l+1}}$ belong to either $\supp(u)$ or $\supp(v)$. Using the fact that $u$ and $v$ are $\tb$-spread monomials, this implies that $i_{2l+1}-i_{2l-1} \geq i_{2l+1}-i_{2l}$ and  $i_{2l+1}-i_{2l-1} \geq i_{2l}-i_{2l-1}$, we obtain the desired conclusion. 
		\end{proof}
		

		Let  $I\subset S$ be an ideal generated by the monomials of same degree. Here, set  $T=\KK[\{t_u : u\in \G(I)\}]$ and $\KK[I]=\KK[u: u \in \G(I)]$. Consider the $\KK$-algebra homomorphism
		\begin{equation*}\label{mapForDefiningIdeal}
			\phi: T \rightarrow \KK[I] \text{ defined by } t_u \mapsto u \text{ for } u\in \G(I).
		\end{equation*}

	
		The kernel of $\phi$ is called the {\em defining ideal} of $\KK[I]$. If $\G(I)$ is a sortable set, then it follows from \cite{S} or 
	 \cite[Theorems 6.15 and  6.16]{VH} that there exists a monomial order $<_{\sort}$ such that the defining ideal of $\KK[I]$ admits the reduced Gr\"obner basis  consisting of binomials of the form $t_u t_v - t_{u'} t_{v'}$,  where $\mathrm{sort}(u,v) = (u',v')$.


		\begin{Corollary}\label{Cor1:CMdomain}
			The $\KK$-algebra $\KK[I(\KT)]$ is a Koszul and Cohen-Macaulay normal domain.
		\end{Corollary}
		\begin{proof}
			As discussed above, with respect to $>_{\sort}$, the Gr\"obner basis of the defining ideal of $\KK[I(\KT)]$ contains quadratic binomials. Due to Fr\"oberg \cite{RF} , we conclude that $\KK[I(\KT)]$ is Koszul and due to a theorem of Sturmfels  \cite{S} we obtain $\KK[I(\KT)]$ is normal, see also \cite[Theorem 5.16]{VH}. Therefore, $\KK[I(\KT)]$ is Cohen-Macaulay domain by \cite[Theorem 1]{MH}.
		\end{proof}
		
		
		Our next goal is to establish  $I(\KT)$ has  the strong persistence property and its powers have linear resolution. Remember  an ideal $I$ is said to satisfy the {\em strong persistence} property if $(I^{k+1} : I) = I^k$ for all $k\geq 1$,  see \cite{HQ} for more information. In addition,  an ideal $I$ is said to satisfy the  {\em persistence property} if:
		\begin{equation*}
			\mathrm{Ass}(I) \subseteq  \mathrm{Ass}(I^2) \subseteq  \cdots \subset  \mathrm{Ass}(I^k) \subseteq  \cdots. 
		\end{equation*}
		In \cite{HQ}, it is proved that an ideal with strong persistence property has the  persistence property.
		
		To achieve our goal, we first recall the definition of $l$-exchange property, see \cite{HHV} or \cite[Sec 6.4]{VH} for more details. Let $T$ and $\phi$ be the same as above and $<$ be a monomial order defined on $T$. A monomial $t_{u_1}t_{u_2}\cdots t_{u_N} \in T$ is called a {\em standard monomial} of $\ker\phi$ with respect to $<$, if $t_{u_1}t_{u_2}\cdots t_{u_N} \notin \ini_<(\ker\phi)$.
	

		\begin{Definition}\label{def:lexchangeprop}
			The monomial ideal $I \subset S$ is said to satisfy the $l$-exchange property with respect to the monomial order $<$ on $T$ if the following two conditions hold: let $t_{u_1}t_{u_2}\cdots t_{u_N} $ and $ t_{v_1}t_{v_2}\cdots t_{v_N} $ be two standard monomials of $\ker \phi$ with respect to $<$ such that
			\begin{enumerate}
				\item[(i)] $\deg_{x_i}$ $u_1 u_2 \cdots u_N = \deg_{x_i}$ $v_1 v_2 \cdots v_N$,  for $i=1,\ldots, q-1$  and $q\leq n-1$,
				\item[(ii)] $\deg_{x_q}$ $u_1 u_2 \cdots u_N$ $<$ $\deg_{x_q}$ $v_1 v_2 \cdots v_N$. 
			\end{enumerate}
			Then there exists some $j$ and $\alpha$ with $q < j \leq n$ such that $x_q u_{\alpha} / x_j \in I$.
		\end{Definition}
	

		\begin{Theorem}\label{lexchangepropertyissatisfied}
			The ideal $I(\KT)$ satisfies the $l$-exchange property with respect to the sorting order $<_{\sort}$.
		\end{Theorem}
		
		\begin{proof}
			Let $t_{u_1}t_{u_2}\cdots t_{u_N} $ and $ t_{v_1}t_{v_2}\cdots t_{v_N} $ be two standard monomials of $\ker \phi$ with respect to $<_{\sort}$ and ${\tb} = (t_1,t_2,\ldots,t_{d-1})$. It can be seen from Proposition~\ref{prop:Sortableset} together with (\ref{Eq:sortingcondition}) that the $N$-tuples with ${\tb}$-spread monomials $(u_1, u_2, \ldots, u_N)$ and $(v_1,v_2, \ldots, v_N)$ are sorted. Assume that the products $u_1u_2 \cdots u_N$ and $v_1v_2 \cdots v_N$ satisfy both conditions in the Definition~\ref{def:lexchangeprop}. The condition (i) together with (\ref{Eq:sortingcondition}) gives
			\begin{equation}\label{Cond:1}
				\deg_{x_i}u_{\gamma} =  \deg_{x_i} v_{\gamma} \text{, for }1\leq i\leq q-1 \text{ and for all } 1 \leq \gamma \leq N,  
			\end{equation}
			and the condition (ii) of Definition \ref{def:lexchangeprop} implies that there exists $\alpha$ with $1\leq \alpha \leq N$ such that 
			\begin{equation}\label{Cond:2}
				\deg_{x_q} u_{\alpha} < \deg_{x_q} v_{\alpha}.
			\end{equation}
			Following (\ref{Cond:1}) and (\ref{Cond:2}), we can write 
	 $$u_{\alpha} = x_{j_1}x_{j_2}\cdots x_{j_p}\cdots x_{j_d}~~\text{and}~~ v_{\alpha} = x_{j_1}x_{j_2}\cdots x_{j_{p-1}} x_q x_{k_{p+1}}\cdots x_{k_d}, $$
			with $j_p > q$. To complete the proof, it is enough to show that $w=x_q u_{\alpha} / x_{j_p} \in I(\KT)$. Note that $q$ and $j_p$ belong to $V_p$. Moreover, $q - j_{p-1}\geq t_{p-1}$ because $v_{\alpha}$ is ${\tb}$-spread and $j_{p+1}-q \geq j_{p+1}-j_p \geq t_p$  because $j_p > q$. This yields that $w$ is a ${\tb}$-spread monomial, as desired. 
		\end{proof}
		

		Let $I = I(\KT)$ and $R=S[\{t_u: u \in \G(I)\}]$. We define a monomial order on $R$ as following: if $u_1,u_2 \in S$ and $v_1, v_2 \in T$, then $u_1 v_1 > u_2 v_2$ if and only if $u_1 >_{\lex} u_2$ or $u_1 = u_2$ and $v_1 >_{\sort} v_2$, where  $>_{\lex}$ denotes  the lexicographical order on $S$ induced by $x_1 > \cdots > x_n$. Let $\mathcal{R}(I) = \oplus_{j\geq 0}I^{j} t^j \subseteq S[t]$ be the Rees ring of $I$. The Rees ring $\mathcal{R}(I)$ has the  following presentation
		\begin{equation*}
			\psi:R= S[\{t_u : u \in \G(I)\}] \rightarrow \mathcal{R}(I),
		\end{equation*}
		with $	x_i \mapsto x_i \text{\;\;for\;\;} 1\leq i \leq n \text{\;\;and\;\; } t_u \mapsto ut \text{\;\;for\;\;} u \in \G(I)$.
		Let $P=\ker \psi$. Then we have the  next  result. 
	

		\begin{Corollary}\label{binomialGrobnerBasis}
			Let $>$ be the monomial order on $R$ as defined above. The reduced Gr\"obner basis of $P$ consists of the binomials of the following form:
			\begin{enumerate}
				\item $t_ut_v - t_{u'}t_{v'}$, where $sort(u,v)= (u',v')$; 
				\item $x_it_u-x_jt_v$, where $i<j$, $x_iu=x_jv$, and $j$ is the largest integer for which $x_i v / x_j \in \G(I)$.
			\end{enumerate}
		\end{Corollary}
		\begin{proof}
			According to \cite[Theorem 5.1]{HHV} (or see \cite[Theorem 6.24]{VH}), it is enough to show that $I(\KT)$ is sortable and satisfies the $\ell$-exchange property with respect to $>_{\sort}$ as noted in Proposition~\ref{prop:Sortableset} and Theorem~\ref{lexchangepropertyissatisfied}.
		\end{proof}
	

		Following the similar argument as in the proof of Corollary~\ref{Cor1:CMdomain}, we obtain the following corollary. 
	
		\begin{Corollary}\label{Cor2:CMdomain}
			The Rees algebra $\mathcal{R}(I(\KT))$ is a normal Cohen-Macaulay domain.
			\end{Corollary}


We are in a position to state  the main result of this section in the next  corollary. 
		\begin{Corollary}\label{strong}
			The ideal $I(\KT)$ satisfies the strong persistence property and all powers of $I(\KT)$ have linear resolution.
		\end{Corollary}
		\begin{proof}
	The strong persistence property of $I(\KT)$ can be deduced  from \cite[Corollary 1.6]{HQ} and Corollary \ref{Cor2:CMdomain}. Moreover, Corollary \ref{binomialGrobnerBasis} together with \cite[Corollary 10.1.8]{HH1} provides that all the powers of $I(\KT)$ have linear resolution, as claimed. 
		\end{proof}


	Here,  we determine the limit depth of $I(\KT)$. By a theorem of Brodmann \cite{BR}, $\depth S/I^k$ is constant for large enough $k$. This constant value is known as the limit depth of  $I$, and denoted by $\lim_{k \rightarrow \infty} \depth S/I^k$. The minimum value of $k$ for which $\depth S/I^k= \depth S/I^{k+t}$ for all $t> 0$ is called the {\em index of depth stability} and denoted by $\mathrm{dstab}(I)$. Let  $\mm$ be the graded maximal ideal of $S$. The analytic spread of an ideal $I\subset S$ is  the Krull dimension of the fiber cone $\mathcal{R}(I)/\mm\mathcal{R}(I)$ and denoted by $\ell(I)$.

		\begin{Definition}[\cite{HQ}, Definition 3.1]\label{def:linear relgraph}
		Let $I\subset S$ be a monomial ideal in $S=K[x_1,\ldots,x_n]$ and $\G(I)=\{u_1,\ldots ,u_r\}$. Then the linear relation graph $\Gamma$ of $I$ is the graph with the edge set
		$$E(\Gamma)=\{\{i,j\}: \ there \ exist \ u_t, u_m\in \G(I) \ such \ that\ x_iu_t=x_ju_m\},$$
		and the vertex set $V(\Gamma)=\bigcup_{\{i,j\}\in E(\Gamma)}\{i,j\}$.
		\end{Definition}
	

	An ideal $I \subset S$ is said to have  {\em linear relations}  if $I$ is generated in degree $d$ and  $\beta_{1,j}(I)=0$ for all $j \neq d+1$.  
We employ  the following lemma to compute $\ell(I(\KT))$.

		\begin{Lemma}(\cite[Lemma 5.2]{D})  \label{lem:analytic spread}
		Let $I$ be a monomial ideal with linear relations generated in a single degree whose linear relation graph 
		$\Gamma$ has $r$ vertices and $s$ connected components. Then $\ell(I)=r-s+1$.
		\end{Lemma}

		We are now ready to  determine the analytic spread of $I(\KT)$ in the following lemma.  
	

		\begin{Lemma}\label{lem:analyspread of KTv}
		Let $\KT$ be a complete $\tb$-spread $d$-partite hypergraph  and  $|V(\KT)|=r$. Then $\ell(I(\KT))=r-d+1$.
		\end{Lemma}
		\begin{proof}
		Let $I=I(\KT)$ and $V=\{V_1, \ldots, V_d\}$. Using Theorem~\ref{thm:linearquotient} and \cite[Lemma 5.2]{D}, it is enough to show that $\Gamma(I)$ has $r$ vertices and $d$ connected components. Let $a_i = \min V_i$ and $b_i= \max V_i$, for all $i=1, \ldots, d$. Let $h,k \in V_i$ for some $i$. Since $\KT$ does not have isolated vertices, this implies that the sets $\{a_1, \ldots, a_d\}$ and $\{b_1, \ldots, b_d\}$ are $\tb$-spread edges in $\KT$. Then $u=x_{a_1}\cdots x_{a_{i-1}}x_h x_{b_{i+1}}\cdots x_{b_d}$ and $v=x_{a_1}\cdots x_{a_{i-1}}x_k  x_{b_{i+1}}\cdots x_{b_d}$ are also $\tb$-spread edges in $\KT$. This shows that $x_k u = x_h v$; hence,  $\{h, k\} \in E(\Gamma)$  and $V(\Gamma)=r$. Moreover, it follows from the definition of $\KT$ that for $i \neq j$ and  $h \in V_i$ and $k \in V_j$, we have the edge $\{h,k\} \notin E(\Gamma)$. Therefore, $\Gamma$ has exactly $d$ connected components, as required. 		
		\end{proof}	


	We now  give the last result of this section in the following theorem.
	
		\begin{Theorem}\label{thm:limitdepth}
		Let $\KT$ be a complete $\tb$-spread $d$-partite hypergraph  and $|V(\KT)|=r$,  and $S$ be the ambient ring of $I(\KT)$. Then $$\lim_{k\to\infty}\depth(S/{I(\KT)^k})=d-1,$$ and $\mathrm{dstab} (I(\KT))\leq  r-d$.
		\end{Theorem}
		
		\begin{proof}
		Let $I=I(\KT)$. Then it follows from Corollary \ref{Cor2:CMdomain} and  a result of Eisenbud and Huneke \cite{EH} that $\lim_{k\to\infty}\depth(S/{I^k})=r-\ell(I)$. From  Lemma~\ref{lem:analyspread of KTv}, we have $r-\ell(I)=r-(r-d+1)=d-1$ as required. In addition, using \cite[Theorem 3.3]{HQ}  and Lemma~\ref{lem:analyspread of KTv}, we see that $\depth(S/I^{r-d})=d-1$. It is shown in \cite[Proposition 2.1]{HHdepth} that if all powers of an ideal have linear resolution,  then $\depth S/I^k \leq \depth S/ I^ t$ for all $k < t$. It follows now from  Corollary~\ref{strong}  that $\mathrm{dstab} (I) \leq r-d$. This completes the proof. 
		\end{proof}


	\section{Normally torsion-free and Cohen-Macaulay $I(\KT)$ } \label{sec:CM}

	In this section, our main goal   is to show that $I(\KT) $ is  normally torsion-free and give a complete characterization of Cohen-Macaulay $I(\KT)$ for $V=\{V_1, \ldots, V_d\}$ such that each $V_i$ is of the form $[a_i,b_i]$ for some integers $a_i, b_i \in \ZZ^+$. To this aim, we begin with the description of minimal prime ideals of $I(\KT)$ and view $\KT$ as a simplicial complex. For more details on simplicial complexes, we refer  the  reader to \cite{HH1}. 
		
		Given a square-free monomial ideal $I\subset R$,  the {\it Alexander dual}  of  $I$, denoted by $I^\vee$ is given by $I^\vee= \bigcap_{u\in \G(I)} (x_i~:~ x_i \in \mathrm{supp}(u))$. The minimal generators of $I^\vee$ correspond to the minimal prime ideals of $I$. Below we give a description of  $\G(I(\KT)^\vee)$. 
	

	\begin{Theorem}\label{TH.ASSPRIMES}
		Let $\KT$ be a complete $\tb$-spread $d$-partite hypergraph with $V(\KT)\subseteq [n]$ and $V=\{V_1, \ldots, V_d\}$.  Furthermore, let $|V_j|=n_j$ with  $V_j=[i_j, i_j+n_j-1]$ for all $j=1, \ldots, d$.  Then $\G(I(\KT)^\vee)$ consists of  the  following monomials:
		\begin{enumerate}
			\item[{\em (i)}]  $\displaystyle\prod_{k\in V_i} x_k$ for all $i=1, \ldots, d$; and,

			\item[{\em (ii)}]  $(\prod_{i=j}^p\prod_{k\in V_i}x_k) / (\prod_{i=j}^{p-1}v_{q_i}  \prod_{i=j+1}^{p}v_{q'_i}  )$, for all $1\leq j<p\leq d$ and for each sequence of nonnegative integers $q_j, \ldots, q_{p-1}$ satisfying
			 \begin{equation}\label{condition:1}
				i_\ell+q'_\ell<i_\ell+n_\ell-1-q_\ell \;\text{ for } j+1 \leq \ell \leq  p-1,
			\end{equation}
			\begin{equation}\label{condition:2}
				i_\ell +q'_\ell -(i_{\ell-1}+n_{\ell-1}-1-q_{\ell-1})=t_{\ell-1}-1
				\text{ for }\ell=j+1, \ldots, p,
			\end{equation}
			where $v_{q_\ell}=\prod_{r=1}^{1+q_\ell}x_{i_\ell+n_\ell-r}, $   for  $\ell=j, \ldots,  p-1$ and $v_{q'_{\ell}}=\prod_{r=0}^{q'_\ell}x_{i_\ell+r},$  for  $\ell=j+1,  \ldots,  p$.
		\end{enumerate}
	\end{Theorem}
		
		\begin{proof}
			Let $\Delta$ be the simplicial complex on $V(\KT)$ such that $I_{\Delta}=I(\KT)$ be the Stanley-Reisner ideal of $\Delta$. Let $\mathcal{F}(\Delta)$ be the set of facets of $\Delta$. For any $F \in \Delta$, we set $x_F=\prod_{i \in F}x_i$. It follows from \cite[Lemma 1.5.4]{HH1} that the standard primary decomposition of $I_{\Delta}$ is given by
			\[
			I_\Delta = \bigcap_{F \in \mathcal{F}(\Delta)}P_{\bar{F}},
			\]
			where $P_{\bar{F}}$ is the monomial prime ideal generated by the variables $x_i$ with $i \in \bar{F}=V(\KT) \setminus F$. Therefore, using \cite[Corollary 1.5.5]{HH1}, it is enough to show that $\mathcal{F}(\Delta)$ is the disjoint union of $\mathcal{F}_1$ and $\mathcal{F}_2$, defined below:
			\begin{enumerate}
	\item [(i)] $\mathcal{F}_1=\{F_1, \ldots, F_d\}$,  where $F_i=\bigcup_{j \neq i,\;j=1}^d V_j$ for all $i=1, \ldots, d$,

				\item[(ii)] For all $1 \leq j < p \leq d$, set $A_{j,p}:=\bigcup_{i \notin\{j, \ldots, p\}, \; i=1}^d V_j$. For each sequence of nonnegative integers $q_j, \ldots, q_{p-1}$  satisfying conditions (\ref{condition:1}) and (\ref{condition:2}), we set 
				\[B_{q_\ell}:=\{i_\ell+n_\ell-1-q_\ell,\ldots,i_\ell+n_\ell-1\} \subsetneq V_\ell
				\text{  for  } \ell=j, \ldots,  p-1,
				\]
				and 
				\[
				B_{q'_\ell}=\{i_\ell,\ldots,i_\ell+q'_\ell\}\subsetneq V_\ell
					\text{  for  } \ell=j+1, \ldots,  p.
				\]
				 Then we get 
				\[\mathcal{F}_2=\{A_{j,p}\;\cup \; (\bigcup_{\ell=j}^{p-1}B_{q_\ell})   \cup (\bigcup_{\ell=j+1}^{p}B_{q'_{\ell}})\:\text{ for all~ } 1 \leq j < p \leq d \text{ and }  q_j, \ldots, q_{p-1}\} .
				\]
			\end{enumerate}
		
			The condition  (\ref{condition:2}) translates into the following: for each $\ell=j,\ldots, p-1$ we have $\max B_{q'_{\ell+1}} - \min B_{q_\ell}= t_{\ell}-1 $. In the construction of elements in $\F_2$, it is enough to determine the integers $q_j, \ldots, q_{p-1}$, because $q'_{\ell}$ is uniquely determined from $q_{\ell-1}$, for all $\ell=j+1, \ldots, p$,   by using the equality in (\ref{condition:2}). 
		
			First, we show that  $\mathcal{F}_1 \subseteq \mathcal{F}(\Delta)$. For any  $F_i \in \mathcal{F}_1$, we have $F_i \cap V_i = \emptyset$. Therefore, $x_{F_i}\notin I_\Delta$. Moreover, for any $k \in V_i$, using the assumption that $\KT$ does not contain any isolated vertices, we obtain that $F_i \cup \{k\}$ contains a $\tb$-spread edge, and hence $x_{F_i}x_{k} \in I_{\Delta}$ and $F_i \in \mathcal{F}(\Delta)$.
			
			Now, assume that  $F\in \mathcal{F}_2$,  where $F=A_{j,p}\;\cup \; (\bigcup_{\ell=j}^{p-1}B_{q_\ell})   \cup (\bigcup_{\ell=j+1}^{p}B_{q'_{\ell}})$ for some $1 \leq j < p \leq d$ and $q_j, \ldots, q_{p-1}$. We here  show that $F\in \Delta$. On contrary, if $x_F\in I_{\Delta}$, then $F$ contains a $\tb$-spread edge, say $G=\{k_1,\ldots,k_d\}$. Then $k_j \in B_{q_j}$ because $G\cap V_j\subseteq F\cap V_j =B_{q_j}$. If $p=j+1$, then by using the condition (\ref{condition:2}), it immediately follows that for any choice of $k_j \in B_{q_j}$, there is no suitable $k_{j+1} \in B_{q'_{j+1}}$ such that $k_{j+1}-k_j \geq t_{j-1}$. If $p > j+1$, then the  condition (\ref{condition:2}) gives that $k_{j+1}\in B_{q_{j+1}}$. Using the condition (\ref{condition:2}) repeatedly in a similar way, we obtain  $k_{p-1}\in  B_{q_{p-1}}$. However, there is no suitable $k_p \in B_{q'_p}$ such that $k_p - k_{p-1} \geq t_{p-1}$, a contradiction. Consequently, we get  $F \in \Delta$. 
			
			In what follows, we demonstrate that  $F \in \mathcal{F}(\Delta)$. Note that 
			$$ V(\KT )\setminus F= (V_{j}\setminus B_{q_j} ) 	\cup (\bigcup_{l=j+1}^{p-1} (V_{l}\setminus (B_{q'_\ell} \cup B_{q_\ell}  ) ) \cup(V_p \setminus B_{q'_p}).$$ 
						Let $a \in V(\KT )\setminus F$. Then $a \in V_{s}$   for some $ j \leq s \leq p$. Set
			\begin{equation*}
				k_r=
				\begin{cases}
					i_r ,& \text{if } r= 1, \ldots, j-1, \\
					i_{r}+n_{r}-1-q_{r} ,&  \text{if } r= j, \ldots,s-1,
					\\
					a ,&  \text{if } r= s,
					\\
					i_{r}+q'_{r} ,&  \text{if } r= s+1, \ldots,p,
					\\
					i_r+n_r-1,&  \text{if } r= p+1 , \ldots, d.
				\end{cases}
			\end{equation*}
			When $s=j$, then we remove the condition on $k_r$ for $r=j, \ldots, s-1$, and similarly, when $s=p$, then we remove the condition on $k_r$ for $r=s+1, \ldots, p$. Using conditions (\ref{condition:1}) and (\ref{condition:2}) together with the assumption that $\Delta$ has no isolated vertices, 
				we obtain that $k_r-k_{r-1} \geq t_{r-1}$  for all $r=2, \ldots, d$. Therefore, $G=\{k_1, \ldots, k_d\} \subseteq F \cup\{a\}$ is 
	a $\tb$-spread edge,  and hence $x_G \in I_\Delta$, as required. 
			
			It remains to check   that $\mathcal{F}(\Delta) \subseteq \mathcal{F}_1 \cup \mathcal{F}_2$. This is equivalent to show that for every face $G$ of $\Delta$ there exists a facet $F \in \mathcal{F}_1 \cup \mathcal{F}_2$ such that $G\subseteq F$. Let $G \in \Delta$ such that $G\cap V_k=U_k$ for all $k=1,\ldots, d$. If $U_k = \emptyset$ for some $k$, then $G \subseteq F_k \in \mathcal{F}_1$. Now, assume that $U_k\neq \emptyset$  for all $k=1, \ldots, d$. Set $a_k=\min U_k$ and $b_k= \max U_k$   for all $k=1, \ldots, d$.  
	In the rest of the proof,   we will use the following fact repeatedly: \\
			
		$(*)$ If there exist $a \in V_\ell$ and $b \in V_{\ell+1}$ such that $b-a < t_\ell$ and $a+t_\ell-1< i_{\ell+1}+n_{\ell+1}-1$, then by letting $ q_\ell= i_\ell+n_\ell-1-a$, and using  the  condition (\ref{condition:2}), there is a unique $q'_{\ell+1}$ such that $b < i_{\ell+1}+q'_{\ell+1}$. \\
			
				{\bf Case(1):}  If there exists some $k$ with  $b_{k+1}-a_k <t_k$,  then it follows from the statement $(*)$  that for a suitable choice of $q_k$ we have $U_k \subseteq   B_{q_k}$ and $U_{k+1} \subseteq B_{q'_{k+1}}$. Since $U_i \subseteq V_i \subset A_{k,k+1}$  for all $i=1, \ldots, k-1, k+2, \ldots,d$, we can  deduce that  $G \subseteq A_{k, k+1}\cup B_{q_k} \cup B_{q'_{k+1}}  \in \F_2$, as desired. 
			
			{\bf Case(2):}   Assume that $b_{k+1}- a_{k} \geq t_k$   for all $k=1, \ldots, d-1$. Since $G \in \Delta$, we know 
	that $G$ does not contain any $\tb$-spread edge. In particular, $\{a_1, \ldots, a_d\} \subseteq G$ is  not a $\tb$-spread edge. This  yields that there exists some $k \in \{2, \ldots, d\}$ for which $a_{k+1}-a_k <t_k$.  We choose minimum $j\geq 1$ for which $a_{j+1}-a_j <t_j$.  Note that $M=\{a_1, a_2, \ldots, a_{j}\}\subset G$ such that, $a_{i+1}-a_i \geq t_i$, for all $i=1, \ldots, j-1$.  In the discussion below, we aim to construct a suitable $F \in \F_2$ such that $G \subset F$. To this aim, we perform the Step $j$ as introduced below.
			
			Step $j$: We set $e_{j}:=a_j$ and $e_{j+1}:=\min \{a \in U_{j+1} : a-e_{j} \geq t_{j}\}$. Note that $ \{a \in U_{j+1} : a-e_{j} \geq t_{j}\}\neq \emptyset$ because $b_{j+1}- a_{j} \geq t_j$. We define $e_{j+r}$ recursively as $e_{j+r}=\min \{a \in U_{j+r} : a-e_{j+r-1} \geq t_{j+r-1}\}$ such that $$\{a \in U_{j+r} : a-e_{j+r-1} \geq t_{j+r-1}\}  \neq \emptyset \text{ for some  } 1 < r < d-j.$$  There exists some $p > j+1$ for which $\{a \in U_{j+r} : a-e_{j+r-1} \geq t\} = \emptyset$, that is, for some $p> j+1$ we have $b_{p}-e_{p-1} < t_{p-1}$, otherwise, $M\cup \{ e_{j+1}, \ldots, e_d\} \subseteq G$ is a $\tb$-spread edge in $G$, a contradiction. Choose minimum $p> j+1$ such that $b_{p}-e_{p-1} < t_{p-1}$.
			
			{\bf Subcase(2.1):}	If for all $j+1 \leq l \leq p-1$ we have $ i_{\ell+1}-e_{\ell} < t_\ell$, then take $c_{\ell+1} \in V_{\ell+1}$ such that  $c_{\ell+1}-e_{\ell} = t_\ell-1$ for $\ell= j, \ldots, p-1$ . This gives us $j, p$ and $q_j, \ldots, q_p$ as described in statement $(*)$  for which $e_\ell \in V_\ell$ and $c_{\ell+1} \in V_{\ell+1}$ with $c_{\ell+1}-e_{\ell}<t_\ell$. Moreover, $U_i \subseteq A_{j,p}$ for all $i \notin \{j, \ldots, p\}$, 
			 and $U_j \subseteq B_{q_j}$, $U_p \subseteq B_{q'_p}$, and $U_\ell \subseteq B_{q_\ell} \cup B_{q'_{\ell}}$  for all $\ell = j+1, \ldots, p-1$. Hence, this implies that 
			 $$G \subseteq A_{j,p}\;\cup \; (\bigcup_{\ell=j}^{p-1}B_{q_\ell})   \cup (\bigcup_{\ell=j+1}^{p}B_{q'_{\ell}}),$$ and we are done. 
			
			{\bf  Subcase(2.2):}  If for some   $j+1 \leq l \leq p-1$, $ i_{\ell+1}-e_{\ell} \geq t_\ell$, then replace $M$ with $M \cup \{ e_{j+1}, \ldots, e_{\ell}, a_{\ell+1}\}\subset G$. In this case, there exists a minimum $j' \geq \ell+1 $ such that $a_{j'+1}-a_{j'} <t_{j'}$. Otherwise,  $M\cup \{a_{\ell+2}, \ldots, a_d\}\subseteq G$ is a $\tb$-spread edge, a contradiction. Repeat Step $j$ by replacing $j$ with $j'$. 
			
		Thanks to  we have finite number of partitions, this process must  be  terminated   after a finite number of steps. If the desired $j$ and $p$
		 are obtained,  then we construct a suitable $F \in \F_2$ with $G \subset F$ as described in Case(2.1). If the desired $j$ and $p$ are not obtained, then $G$ contains a $\tb$-spread edge  in $G$, a contradiction.
		\end{proof}
	

		We illustrate the construction of monomials of the forms (i) and (ii) in Theorem~\ref{TH.ASSPRIMES} in the following example.

			\begin{Example}{\em
		Let $V=\{V_1,V_2,V_3,V_4\}$ with $V_1=[1,2]$, $V_2=[4,6]$, $V_3=[8,10]$, $V_4=[12,13]$,  and $\tb=(3,4,3)$.  
	One can easily see that the minimal generators of the edge ideal of $\KT$  are  as follows:
	
	\begin{center}
		\begin{tabular}{ c c c }
			$x_1x_4x_8x_{12}$\quad&\quad\\
			$x_1x_4x_8x_{13}$\quad&\quad\\
			$x_1x_4x_9x_{12}$\quad&\quad\\
			$x_1x_4x_9x_{13}$\quad&\quad\\
			$x_1x_4x_{10}x_{13}$\quad&\quad\\
			$x_1x_5x_9x_{12}$\quad&\quad$x_2x_5x_9x_{12}$\\
			$x_1x_5x_9x_{13}$\quad&\quad$x_2x_5x_9x_{13}$\\
	        $x_1x_5x_{10}x_{13}$\quad&\quad$x_2x_5x_{10}x_{13}$\\
			$x_1x_6x_{10}x_{13}$\quad&\quad$x_2x_6x_{10}x_{13}$
		\end{tabular}
	\end{center}

	Following Theorem~\ref{TH.ASSPRIMES},  the minimal generators of $I(\KT)^\vee$ are given as follows: 
		\begin{enumerate}
		\item[ (i)] The monomials of the  form (i) described in Theorem~\ref{TH.ASSPRIMES} are $x_1x_2,x_4x_5x_6,$ $x_8x_9x_{10},$ 
	and $x_{12}x_{13}$. 
		\item[(ii)] The construction of  monomials of  the form (ii) described in Theorem~\ref{TH.ASSPRIMES} is given in the following table.
		\bigskip
		
		\begin{tabular}{|c|c|l|c|}
		\hline
		$j$ & $p$ & $q_j,\ldots,q_{p-1}$, $q'_{j+1},\ldots,q'_{p}$& $u$ \\
		\hline
		$1$ & $2$ & $q_1=0$, $q'_2=0$ & $\; x_1x_5x_6\;$ \\
		\hline
		$1$ & $3$ & $q_1=0$, $q'_2=0$, $q_2=0$, $q'_3=1$ & $\; x_1x_5x_{10}\;$\\
		
		   &      & $q_1=0$, $q'_2=0$, $q_2=1$, $q'_3=0$ & $\; x_1x_9x_{10}\;$\\
		\hline
		$1$ & $4$ & $q_1=0$, $q'_2=0$, $q_2=0$, $q'_3=1$,   $q_3=0$ , $q'_4=0$& $\; x_1x_5x_{13}\;$\\
		
		    &     & $q_1=0$, $q'_2=0$, $q_2=1$, $q'_3=0$, $q_3=0$, $q'_4=0$ & $\; x_1x_9x_{13}\;$\\
		\hline
		$2$ & $3$ & $q_2=0 $, $q'_3=1$ & $\; x_4x_5x_{10}\;$\\
		
		    &    & $q_2=1 $, $q'_3=0$ & $\; x_4x_9x_{10}\;$\\
		\hline
		$2$ & $4$ & $q_2=0$, $q'_3=1$, $q_3=0$, $q'_4=0$ & $\; x_4x_5x_{13}\;$\\
		    &     & $q_2=1$, $q'_3=0$, $q_3=0$, $q'_4=0$ & $\; x_4x_9x_{13}\;$\\
		\hline
		$3$ & $4$ & $q_3=0$, $q'_4=0$ & $\; x_8x_{9}x_{13}$\\
		\hline
		\end{tabular}
		\end{enumerate}
\bigskip

Accordingly, we get 
\begin{align*}
\mathrm{Ass}(I(\KT))=\{&(x_1, x_2), (x_4, x_5, x_6), (x_8, x_9, x_{10}), (x_{12}, x_{13}), (x_1,x_5,x_6), (x_1,x_5,x_{10}),\\
&  (x_1,x_9,x_{10}), (x_1,x_5,x_{13}), (x_1,x_9,x_{13}),   (x_4,x_5,x_{10}),  (x_4,x_9,x_{10}), \\
& (x_4,x_5,x_{13}), (x_4,x_9,x_{13}), (x_8,x_{9},x_{13})\}.
\end{align*}
}
	\end{Example}


As an immediate consequence of Theorem~\ref{TH.ASSPRIMES}, we obtain the following corollary, which will be used to prove the 
normally torsion-freeness of $I(\KT)$.

	\begin{Corollary} \label{V-Structure}
		Let $\KT$ be a complete $\tb$-spread $d$-partite hypergraph with $V(\KT)\subseteq [n]$ and $V=\{V_1, \ldots, V_d\}$.  Furthermore, let $|V_j|=n_j$   with  $V_j=[i_j, i_j+n_j-1]$ for all $j=1, \ldots, d$.     If $v:=\prod_{j=1}^{d} x_{i_j}$, then $v\in \mathfrak{p}\setminus \mathfrak{p}^2$ for all $\mathfrak{p}\in \mathrm{Min}(I(\KT))$. 
	\end{Corollary}
 \begin{proof}
Let $v=\prod_{j=1}^{d} x_{i_j}$. The minimal prime ideals of $I=I(\KT)$ correspond to the minimal generators of $I^\vee$ described in statements  (i) and (ii) of Theorem~\ref{TH.ASSPRIMES}. The minimal primes corresponding to the generators of the form (i) are $ \mathfrak{p}_i= (x_k: k \in V_i)$ and $v \notin  \mathfrak{p}_i^2$ for all $i=1, \ldots, d$. Moreover, each generator of $I^\vee$ of the form (ii) is constructed by fixing $j$, $p$ and $q_j, \ldots, q_p$. Let $ \mathfrak{q}$ be a minimal prime of $I$ corresponding to a generator of  the form (ii). Then  $x_{i_k} \in  \mathfrak{q}$ if and only if $k=j$, 
 as required. 
 \end{proof}


We recollect the following lemma,   which will be used repeatedly in the next proposition and Theorem \ref{Normally torsion-freeness}.

\begin{Lemma}\label{LEM. Multipe}{\em(\cite[Lemma  3.12]{SN}) }
	Let  $I$ be a monomial ideal in a polynomial ring $S=\KK[x_1, \ldots, x_n]$ with 
	$\G(I)=\{u_1, \ldots, u_m\}$, and $h=x_{j_1}^{b_1}\cdots x_{j_s}^{b_s}$ with $j_1, \ldots, j_s \in \{1, \ldots, n\}$ be a monomial in $S$. Then  $I$ is normally torsion-free  if and only if $hI$ is normally torsion-free. 
\end{Lemma}


In order to establish Theorem \ref{Normally torsion-freeness}, we require  the following auxiliary  proposition.
 For a given square-free monomial ideal $I \subset \KK[x_1, \ldots, x_n]$, we denote by $I \setminus x_i$  the ideal generated by 
 those elements in $\G(I)$ that do not contain $x_i$ in their support. 
	
	\begin{Proposition}\label{Case.2d}
		
			Let $\KT$ be a complete $\tb$-spread $d$-partite hypergraph with $V(\KT)\subseteq [n]$ and $V=\{V_1, \ldots, V_d\}$. 
	Furthermore, let  $|V_j|=2$   with 		 $V_j=\{i_j, i_j+1\}$ for all $j=1, \ldots, d$. Then $I(\KT)$  is normally torsion-free. 
	\end{Proposition}
	\begin{proof}
To simplify the notation, set $I:=I(\KT)$. 
		We proceed by induction on $d$. If $d=1$,  then there is nothing to show. Hence, assume that $d>1$ and that the result holds for any 
		complete $\tb$-spread $(d-1)$-partite hypergraph.  Choose an arbitrary element $\mathfrak{p}\in \mathrm{Min}(I)$ and 
		set $v:=\prod_{j=1}^{d} x_{i_j}$. It follows at once from Corollary \ref{V-Structure} that    $v\in \mathfrak{p}\setminus \mathfrak{p}^2$. We show that 
		$I\setminus x_r$ is normally torsion-free for each $x_r \in \mathrm{supp}(v)$.  Without loss of generality, we let $V_1=\{1,2\}$ and we prove that $I\setminus x_1$ is normally torsion-free. It is not hard to check that $I\setminus x_1=x_2L$ where $L$ is the edge ideal of $\tb$-spread $d$-partite hypergraph with vertex partition $\V'=\{V'_2, \ldots, V'_d\}$ such that,  for all $i=2, \ldots, d$, the set $V'_i$ is obtained from $V_i$ after removing the isolated vertices, if any. One can conclude  from the inductive hypothesis that $L$ is normally torsion-free. Here, using Lemma~\ref{LEM. Multipe} implies that  
		$I\setminus x_1$ is normally torsion-free.  It follows now from \cite[Theorem 3.7]{SNQ}  that $I$ is normally torsion-free, as claimed. 
		\end{proof}

	
	\begin{Theorem} \label{Normally torsion-freeness}
		Let $\KT$ be a complete $\tb$-spread $d$-partite hypergraph with $V(\KT)\subseteq [n]$ and $V=\{V_1, \ldots, V_d\}$.  Furthermore, let $|V_j|=n_j$ with  $V_j=[i_j, i_j+n_j-1]$ for all $j=1, \ldots, d$. 		Then $I(\KT)$ is normally torsion-free. In particular, 	$I(\KT)$ is normal. 
	\end{Theorem}
	
	\begin{proof}
		We first assume that $|V_j|=1$ for some $1\leq j \leq d$, say $V_j=\{z\}$. Let $I=I(\KT)$. Then we can write $I=x_zL$ such that  $L$  can be viewed as the edge ideal associated to a complete $\tb$-spread $(d-1)$-partite hypergraph. According to Lemma \ref{LEM. Multipe}, $I$ is normally torsion-free if and only if $L$ is normally torsion-free. Thus, we reduce to the case $|V_j|\geq 2$ for all  $j=1, \ldots, d$. Set $v:=\prod_{j=1}^{d}x_{i_j}$. Pick an arbitrary element $\mathfrak{p}\in \mathrm{Min}(I)$. One can derive from Corollary \ref{V-Structure}  that 
		$v\in \mathfrak{p}\setminus \mathfrak{p}^2$. To complete the proof, it is sufficient to establish $I\setminus x_s$ in normally torsion-free for each  
		$x_s \in \mathrm{supp}(v)$.  To accomplish this, we use the induction on $n:=|V(\KT)|$. On account of $|V_j|\geq 2$ for all  $j=1, \ldots, d$, this implies that $n\geq 2d$. 
		The case in which $n=2d$ can be deduced  according to Proposition \ref{Case.2d}. 
		Now, suppose that $n>2d$.  It is not hard to see that $I\setminus x_s$ is again the  edge ideal of  the  $\tb$-spread $d$-partite hypergraph obtained from  $\KT$ by removing all the edges  that contain $s$. One can deduce from the inductive hypothesis that $I\setminus x_s$ is normally torsion-free. Here, in view of  \cite[Theorem 3.7]{SNQ}, we conclude that $I$ is normally torsion-free, as desired.   
	
	The last assertion can be deduced according to \cite[Theorem 1.4.6]{HH1}.  
	\end{proof}


We can readily provide  the following corollary inspired by Theorem \ref{Normally torsion-freeness}.   A {\em matching} in a hypergraph $\H$ is a family of pairwise disjoint edges, and the maximum cardinality of a matching is denoted by $\nu (\H)$. The transversal number of a hypergraph $\H$, denoted by $\tau(\H)$ is the minimal cardinality of a transversal of $\H$.  A hypergraph $\H$ is said to satisfy the K\"onig property if $\nu (\H)=\tau(\H)$, see \cite[Chapter 2, Section 4]{B}. 

\begin{Corollary} \label{Konig-Property}
Let $\KT$ be a complete $\tb$-spread $d$-partite hypergraph.  Then  $I(\KT)$ satisfies  the K\"onig property.
\end{Corollary}

\begin{proof}
Based on Theorem \ref{Normally torsion-freeness}, we get  $I(\KT)$  is normally torsion-free. In addition, by virtue of  \cite[Theorem 14.3.6]{VI2}, one can deduce that $\KT$ has the max-flow min-cut  property. It follows now  from \cite[Corollary 14.3.18]{VI2} that $\KT$  has the packing property. On the other hand, by virtue of    \cite[Definition 2.3]{HM}, we obtain  $I(\KT)$  satisfies the   K\"onig property. This completes the proof. 
\end{proof} 
	
		
Next, we give a  characterization of Cohen-Macaulay $I(\KT)$. To do this, we first determine the height of  $I(\KT)$.
	\begin{Proposition}\label{Prop:height}
		Let $\KT$ be a complete $\tb$-spread $d$-partite hypergraph with $V(\KT)\subseteq [n]$ and $V=\{V_1, \ldots, V_d\}$.  Furthermore, let $|V_j|=n_j$  with  $V_j=[i_j, i_j+n_j-1]$ for all $j=1, \ldots, d$. Then $\mathrm{ht}(I(\KT))=\min\{n_1, \ldots, n_d\}$, where $\mathrm{ht}(I(\KT))$ denotes the height of $I(\KT)$.
	\end{Proposition}
	
	\begin{proof}
		Let $I:=I(\KT)$ and $n_k=\min\{n_1, \ldots, n_d \}$. Since $\KT$ does not contain any isolated vertices,  this  yields that 
\begin{equation}\label{eq:konig}
	\{i_1,\ldots,i_d\}, \{i_1 +1,\ldots,i_d +1\}, \ldots, \{i_1+n_k-1,\ldots,i_d+n_k-1\},
\end{equation} are pairwise disjoint $\tb$-spread edges in $\KT$.  Hence, we obtain the following monomials  $$x_{i_1}x_{i_2}\ldots x_{i_d}, x_{i_1 +1}x_{i_2 +1}\ldots x_{i_d +1},\ldots,x_{i_1 +n_k-1}x_{i_2 +n_k-1}\ldots x_{i_d +n_k-1}$$
	belong to  $\G(I)$. This gives that $\mathrm{ht}(I)\geq n_k$. It follows also from  Theorem~\ref{TH.ASSPRIMES} that  $(x_i : i \in V_k)$ is a minimal prime of $I$ with height $n_k$. This finishes our proof.  
	\end{proof}
	
	Note that the K\"onig property of $\KT$ can be also observed from the proof of above proposition. Indeed, the inequality $\nu (\H)\leq \tau(\H)$ holds for any hypergraph $\H$ and the $\tb$-spread edges given in (\ref{eq:konig}) give a maximal matching in $\KT$.
	

	Under the assumptions of  Theorem~\ref{TH.ASSPRIMES}, one can compute the degree of generators of $I^\vee=I(\KT)^\vee$. It is easy to see that $\deg \prod_{k \in V_i} x_k = n_i$  for all $i=1, \ldots, d$. Now, let $u \in \G(I^\vee)$ of  the  form (ii) for some $1 \leq j < p \leq d$ and $q_j, \ldots, q_p$. Then $u= (\prod_{i=j}^p\prod_{k\in V_i}x_k) / (\prod_{i=j}^{p-1}v_{q_i}  \prod_{i=j+1}^{p}v_{q'_i}  )$. Let $h$ be the product of variables with indices in $[i_j, i_p+n_p-1] \setminus (V_j \cup \cdots \cup  V_p)$ and $w=(uh)/h$.  Then $\deg w= \deg u$. 
	
	We have $\deg h (\prod_{i=j}^p\prod_{k\in V_i}x_k) = (i_p +n_p-1)-i_j+1$. Moreover, it follows from the  condition (\ref{condition:2}) that  $\deg (h  \prod_{i=j}^{p-1}v_{q_i}  \prod_{i=j+1}^{p}v_{q'_i}  ) =\sum_{i=j}^{p-1} t_i$. We thus get  
$$\deg w= (i_p +n_p-1)-i_j+1-\sum_{i=j}^{p-1} t_i= i_p-i_j+n_p-\sum_{i=j}^{p-1} t_i.$$ 
 Hence, we obtain 
	\begin{equation}\label{eq:deg}
		\deg u = i_p-i_j+n_p-\sum_{i=j}^{p-1} t_i.
	\end{equation}
	
		A square-free monomial ideal is said to be {\em unmixed} if its minimal prime ideals are of the same height. Using the description of generators of $I(\KT)^\vee$ and their degrees, we obtain the following characterization for unmixedness of $I(\KT)$.
	\begin{Theorem}\label{thm:unmixed}
	Let $\KT$ be a complete $\tb$-spread $d$-partite hypergraph with $V(\KT)\subseteq [n]$ and $V=\{V_1, \ldots, V_d\}$.  Furthermore, let $|V_j|=n_j$ 
with  $V_j=[i_j, i_j+n_j-1]$ for all $j=1, \ldots, d$.  Then $I(\KT)$ is unmixed if and only if $n_1=\cdots=n_d=s$, and  for each $j=1, \ldots, d-1$ either $i_{j+1} - (i_{j}+s-1) > t_j-1$  or $i_{j+1}-i_j=t_j$. 
	\end{Theorem}
	
	\begin{proof}
		Let $I=I(\KT)$ be  unmixed. Then every minimal prime of $I$ has  the  same height, equivalently, $I^\vee$ is generated in the same degree. By Theorem~\ref{TH.ASSPRIMES}, we know that every $V_j$ corresponds to a minimal generator in $I^\vee$,  and this yields $n_1=\cdots =n_d$. Let $n_1=\cdots=n_d=s$. We only need to show that for each $j=1, \ldots, d-1$ either $i_{j+1} - (i_{j}+s-1) > t_j-1$  or $i_{j+1}-i_j=t_j$. Indeed, if  $i_{j+1} - (i_{j}+s-1) \leq t_j-1$ for some $j$, then we obtain $u \in \G(I^\vee)$ of the form (ii) with $p=j+1$ and a suitable choice of $q_{j}$ and $q'_{j+1}$ as described in statement  $(*)$  in the proof of Theorem~\ref{TH.ASSPRIMES}. It follows from (\ref{eq:deg}) that $\deg u=  i_{j+1}-i_j+s-t_j$. Since $\deg u=s$, we obtain $i_{j+1}-i_j=t_j$. 
		
		Now,  assume that  for all $j=1, \ldots, d$  we have  $n_j=s$ and  for each $j=1, \ldots, d-1$   either $i_{j+1} - (i_{j}+s-1) > t_j-1$  or $i_{j+1}-i_j=t_j$. Then all generators of $I^\vee$ of the form (i) have same degree $s$. If $I^\vee$ has no generator of the form (ii), then the proof is complete. Otherwise, 
		let $u \in \G(I^\vee)$ of the form (ii) for some $j, p$ and $q_j \ldots, q_{p-1}$. Then, for all $\ell =j , \ldots, p-1 $,
	 we have  $i_{\ell+1}-i_\ell=t_{\ell}$, because if  $i_{\ell+1} - (i_{\ell}+s-1) > t_{\ell}-1$ for some $\ell$,   then $q_{\ell}$ and $q'_{\ell+1}$ do not satisfy the  condition (\ref{condition:2}). This gives that $i_p= i_j+\sum_{i=j}^{p-1} t_i$. Using (\ref{eq:deg}), we obtain  
	$$\deg u =  i_p-i_j+s-\sum_{i=j}^{p-1} t_i= i_j+\sum_{i=j}^{p-1} t_i-i_j+s-\sum_{i=j}^{p-1} t_i =s,$$ and the proof is done.  
	\end{proof}


\begin{Remark} {\em 
Let $V=\{V_1,V_2,V_3,V_4\}$ with $V_1=[2,4]$, $V_2=[6,8]$, $V_3=[9,11]$, $V_4=[13,15]$,  and $\tb=(2,3,4)$. 
By virtue of  Theorem~\ref{thm:unmixed}, the edge ideal $I=I(\KT)$ is unmixed. In fact, by using Theorem~\ref{TH.ASSPRIMES}, the minimal primes of 
$I$ are  as follows:
\begin{align*}
\mathrm{Ass}(I)=\{&(x_2,x_3,x_4), (x_6,x_7,x_8), (x_9,x_{10},x_{11}), (x_{13},x_{14},x_{15}), (x_6,x_7,x_{11}), \\
& (x_6,x_7,x_{15}),   (x_6,x_{10},x_{11}), (x_6,x_{10},x_{15}), (x_6,x_{14},x_{15}),  (x_9,x_{10},x_{15}), \\
& (x_9,x_{14},x_{15})\}. 
\end{align*}
 However, one can verify with {\em Macaulay2 }
\cite{M2} that $S/I$ is not Cohen-Macaulay. }
\end{Remark}


The above remark  states that  unmixedness is not a sufficient for the edge ideal of $\tb$-spread $d$-partite hypergraphs being Cohen-Macaulay. 
In what follows, we give a characterization of $\KT$ with Cohen-Macaulay edge ideals. To do this, we introduce the  following notations, that is, 
 $q(u_k):= |\set(u_k)|$ and $q(I):= \max\{q(u_1), \ldots, q(u_r)\}$. 
\bigskip

We are in a position to state the last result of this section in the subsequent  theorem.

	\begin{Theorem}\label{CM}
		Let $\KT$ be a complete $\tb$-spread $d$-partite hypergraph with $V(\KT)\subseteq [n]$ and $V=\{V_1, \ldots, V_d\}$.  Furthermore, let $|V_j|=n_j$ with  $V_j=[i_j, i_j+n_j-1]$ for all $j=1, \ldots, d$.  Then $S/I(\KT)$ is Cohen-Macaulay if and only if  either $I(\KT)$ is a principal ideal, or  $n_1=\cdots= n_d=s$ and $i_{j+1}-i_{j}=t_j$ for each $j=1, \ldots, d-1$. 
	\end{Theorem}
	
	\begin{proof}
		Let $I=I(\KT)$ and $S$ be the ambient ring of $I$.  Since $I$ has linear quotients, thanks to Theorem~\ref{thm:linearquotient},  it follows from \cite[Corollary 1.6]{HT} that the length of the minimal free resolution of $S/I$ over $S$ is equal to $q(I) + 1$. This implies that  $\depth (S/I) =|V(\KT)|-q(I)-1$. Moreover, $\dim (S/I)= |V(\KT)|-\mathrm{ht}(I)$, where $\mathrm{ht}(I)$ denotes  the height of $I$. 
	 This summarizes to $S/I$ is Cohen–Macaulay if and only if $\mathrm{ht}(I) = q(I) + 1$. 
	  Therefore,  it is enough to show that $\mathrm{ht}(I)  = q(I) + 1$ if and only if $n_1=n_2=\cdots= n_d=s$ and $i_{j+1}-i_{j}=t_j$ for each $j=1, \ldots, d-1$. 
		
		If $I$ is a principal ideal then $S/I$ is Cohen Macaulay. Now, assume $n_1=n_2=\cdots= n_d=s$ and $i_{j+1}-i_{j}=t_j$ for each $j=1, \ldots, d-1$. Let $u=x_{k_1}\cdots x_{k_d} \in \G(I)$, 
		 where $k_i \in V_i$   for all $i=1, \ldots, d$. Since $[i_1, k_1-1] \subseteq  V_1$ and $[k_{j-1}+t_{j-1}, k_j-1]\subseteq  V_j$ for all $j=2, \ldots, d$, by Proposition~\ref{prop:set}, we obtain $q(u)=k_d-i_1-\sum_{j=1}^{d-1}t_j$. This shows that the maximum value of $q(u)$ is obtained when $k_d$ takes the maximum possible value which is $\max V_d=i_d+s-1$. Furthermore, using $i_{j+1}-i_j=t_j$ for all $j=1,\ldots,d-1$,  this gives that  $i_d=i_1+\sum_{j=1}^{d-1}t_j$. Hence,  we have   $q(I)=s-1$, as required.

		Conversely, suppose $S/I$ is Cohen-Macaulay, that is, $\mathrm{ht}(I)  = q(I) + 1$.
	 It follows from $\mathrm{ht}(I)  = q(I) + 1$ that  $I$ is unmixed and by using Proposition \ref{thm:unmixed}, 
	this  yields that, for all $j=1, \ldots, d$,  we have  $n_j=s$, and  for each $j=1, \ldots, d-1$ either $i_{j+1} - (i_{j}+s-1) > t_j-1$  or $i_{j+1}-i_j=t_j$. Then $\mathrm{ht}(I)=s$  thanks to Proposition~\ref{Prop:height}. If $s=1$, then $I$ is a principal ideal. Now, let $s>1$.  We only need to show that, for each $j=1, \ldots, d-1$, we have $i_{j+1}-i_j=t_j$. Suppose that for some $j$ we have $i_{j+1} - (i_{j}+s-1) > t_j-1$ . Let $v=x_{i_1+s-1}x_{i_2+s-1}\cdots x_{i_d+s-1}$. Then $v \in \G(I)$ because $\KT$ do not contain isolated vertices and $\{i_1+s-1, i_2+s-1, \ldots, i_d+s-1\}$ is a $\tb$-spread edge in $\KT$. Now, 
	 Proposition~\ref{prop:set} gives that $ \set(v)\cap V_1 = [i_1, i_1+s-2]$ and $ \set(v)\cap V_{j+1} = \{i_{j+1}, \ldots, i_{j+1}+s-2\}$. This shows that $q(v) > 2(s-1)$ and $q(I)+1 > \mathrm{ht}(I)=s$, a contradiction.
	\end{proof}



\end{document}